\documentclass[12pt]{article}
\usepackage[T1]{fontenc}
\usepackage{amssymb,amsmath,amsthm}
\usepackage{epigraph}

\RequirePackage{ifpdf}
\ifpdf
   \usepackage[pdftex]{hyperref}
\else
   \usepackage[hypertex]{hyperref}
\fi

\usepackage{enumerate,xspace}

\newtheorem{theorem}{Theorem}[section]
\newtheorem{lemma}[theorem]{Lemma}
\newtheorem{proposition}[theorem]{Proposition}
\newtheorem{corollary}[theorem]{Corollary}

\theoremstyle{definition}
\newtheorem{definition}[theorem]{Definition}

\newtheorem{remark}[theorem]{Remark}

\newcommand{\nc}{\newcommand}
\nc{\A}{\mathcal{A}}
\nc{\Null}{\mathbf{0}}
\nc{\Eins}{\mathbf{1}}
\nc{\Nps}{\mathrm{M}^N_\infty}

\nc{\DMO}{\DeclareMathOperator}
\DMO{\acl}{acl}
\DMO{\cl}{cl}
\DMO{\IND}{\;ind\;}
\DMO{\gate}{gate}
\DMO{\etag}{etag}
\DMO{\cont}{cont}
\DMO{\boundary}{bd}
\DMO{\tp}{tp}
\DMO{\ab}{a}
\DMO{\ba}{b}

\def\Ind#1#2{#1\setbox0=\hbox{$#1x$}\kern\wd0\hbox to 0pt{\hss$#1\mid$\hss}
\lower.9\ht0\hbox to 0pt{\hss$#1\smile$\hss}\kern\wd0}
\def\Notind#1#2{#1\setbox0=\hbox{$#1x$}\kern\wd0\hbox to
0pt{\mathchardef\nn="0236\hss$#1\nn$\kern1.4\wd0\hss}\hbox
to 0pt{\hss$#1\mid$\hss}\lower.9\ht0
\hbox to 0pt{\hss$#1\smile$\hss}\kern\wd0}

\begin{document}

\title{An alternative axiomization of $N$-pseudospaces}
\date{1 May 2017}

\author{Katrin Tent and Martin Ziegler}
\maketitle

\begin{abstract}
  We give a new axiomatization of the $N$-pseudospace, studied in
  \cite{kT14} and \cite{BMPZ13}, based on the zigzags introduced in
  \cite{kT14}. We also present a more detailed account of the
  characterization of forking given in \cite{kT14}.
\end{abstract}


\section{Introduction}

Pseudospaces were studied as buildings in \cite{BMPZ13} and in
\cite{kT14} from different points of view, the first one focussing on
the Weyl group, the latter one on a graph theretic approach.

The crucial tool in \cite{kT14} for describing algebraic closure and
forking in these geometries were equivalence classes of so-called
\emph{reduced paths} between vertices where two paths are equivalent
if they change direction (in the sense of the partial order underlying
the pseudospace) in the same vertices.

We now take a more abstract approach. We consider the equivalence
classes of reduced paths as \emph{zigzags} in a lattice and give a new
axiomatization of the pseudospaces in this terminology. Some of the
results and their proofs become somewhat easier in this context. In
particular, we give a more detailed account of the forking
characterization given in \cite{kT14} and prove some more properties
which make it easier to work with.


\section{Simply connected lattices}
Let $V$ be a lattice.\\

\begin{definition}
  An \emph{alternating sequence} of length $n$ is a sequence
  $x_0\ldots x_n$ of elements of $V$ such that $x_i\leq x_{i+1}$ or
  $x_i\geq x_{i+1}$ alternatingly. We will write such a sequence in
  the form $a_0,b_0.a_1,b_1\ldots$ or $b_0,a_1,b_1\ldots$,
  respectively, if $a_i\leq b_i$ and $b_i\geq a_{i+1}$. The $a_i$ are
  the \emph{sinks}, the $b_i$ the \emph{peaks} of the sequence.

   A \emph{zigzag} is an alternating sequence where
   $a_i=\inf(b_{i-1},b_i)$ and\linebreak $b_i=\sup(a_i,a_{i+1})$ and
   furthermore $a_i\not=b_i$ and $b_i\not=a_{i+1}$, for all $i$ for
   which this makes sense. An alternating sequence $x_0\ldots x_n$ of
   length $\geq 2$ is a \emph{weak zigzag}, if for all $i$ the element
   $x_i$ is incomparable with $x_{i+2}$ and with $x_{i+3}$.
\end{definition}

Zigzags were introduced by the first author in \cite{kT14}.

\begin{remark}\label{R:cycle}
 It is easy to see that every zigzag of length $\geq 2$ is also a weak
 zigzag.
 $x_{i+4}$ are not comparable, since $a_j\leq a_{j+2}$, for example,
 would imply $a_j\leq b_{j+1}$.
\end{remark}

An alternating sequence $x'_0\ldots x'_n$ is a \emph{refinement} of
$x_0\ldots x_n$ is $x'_0=x_0$, $x'_n=x_n$, $a_i\leq a'_i$ and
$b'_i\leq b_i$. Clearly, a sequence which satisfies
$a_i=\inf(b_{i-1},b_i)$ and $b_i=\sup(a_i,a_{i+1})$ has no proper
refinement.\\

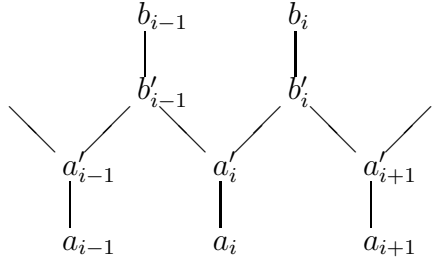
\noindent\begin{center}
\begin{picture}(60,30)(-30,-15)
  \put(-10,15){$b_{i-1}$}
  \put(-9,14){\line(0,-1){6}}
  \put(10,15){$b_i$}
  \put(11,14){\line(0,-1){6}}

  \put(-10,5){$b'_{i-1}$}
  \put(10,5){$b'_i$}

  \put(-20,-5){$a'_{i-1}$}
  \put(-21,-2){\line(-1,1){6}}
  \put(-17,-2){\line(1,1){6}}
  \put(0,-5){$a'_i$}
  \put(-1,-2){\line(-1,1){6}}
  \put(3,-2){\line(1,1){6}}
  \put(20,-5){$a'_{i+1}$}
  \put(19,-2){\line(-1,1){6}}
  \put(23,-2){\line(1,1){6}}

  \put(-20,-15){$a_{i-1}$}
  \put(-19,-12){\line(0,1){6}}
  \put(0,-15){$a_i$}
  \put(1,-12){\line(0,1){6}}
  \put(20,-15){$a_{i+1}$}
  \put(21,-12){\line(0,1){6}}
\end{picture}

\end{center}

\begin{lemma}
  Let $x'_0\ldots x'_n$ be a refinement of $x_0 \ldots x_n$. Then the
  following holds:
  \begin{enumerate}
  \item If $x_0\ldots x_n$ is a weak zigzag, then also $x'_0\ldots
    x'_n$ is a weak zigzag.
  \item For every $i$, if $a_i=\inf(b_{i-1},b_i)$ then also
    $a'_i=\inf(b'_{i-1},b'_i)$ and, dually, if
    $b_i=\sup(a_i,a_{i+1})$, then also $b'_i=\sup(a'_i,a'_{i+1})$.
  \end{enumerate}
\end{lemma}
\begin{proof}
  Easy.
\end{proof}

Hence we obtain:

\begin{corollary}\label{C:weak_zigzag}
  Every weak zigzag can be refined to a zigzag.
\end{corollary}
\begin{proof}
  Set $a'_i=\inf(b_{i-1},b_i)$ and $b'_i=\sup(a'_i,a'_{i+1})$. Note
  that a weak zigzag which satisfies $a_i=\inf(b_{i-1},b_i)$ and
  $b_i=\sup(a_i,a_{i+1})$ is a zigzag.
\end{proof}

\begin{lemma}\label{L:davorsetzen}
  Let $a_0,b_0,x_2 \ldots x_n$ be a zigzag and $c$ an element with
  $c\geq a_0$ and $c\not\geq b_0$. Then,
  \begin{enumerate}
  \item if $c < b_0$, the sequence $c,b_0,x_2\ldots x_n$ is a zigzag,
  \item or, if $c$ and $b_0$ are incomparable, the sequence
    $c,a_0,b_0,x_2\ldots x_n$ is a weak zigzag.
  \end{enumerate}
\end{lemma}
\begin{proof}
  This is easy to check.
\end{proof}
\noindent Note that in case 2 $c,a_0,b_0,x_2\ldots x_n$ refines to the
zigzag $c,\inf(c,b_0),b_0,x_2\ldots x_n$.

\begin{definition}
  A (weak) \emph{zigzag cycle} is a closed (weak) zigzag of length
  $2n$ which satisfies the zigzag condition (considering indices
  modulo $2n$). Thus a zigzag cycle is of the form $a_0,b_0\ldots
  b_{n-1},a_n=a_0$ where $a_0=a_n=\inf(b_{n-1},b_0)$, or of the form
  $b_0,a_1\ldots a_n,b_n$, where $b_0=b_n=\sup(a_n,a_1)$.

\end{definition}
By Remark~\ref{R:cycle}, a weak zigzag cycle has length at least
$6$:\\

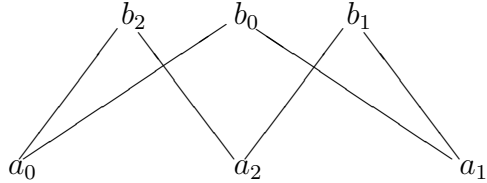
\noindent\begin{centering}
\begin{picture}(60,30)(-30,-15)
  \put(-15,10){$b_2$}
  \put(0,10){$b_0$}
  \put(15,10){$b_1$}

  \put(-30,-10){$a_0$}
  \put(-28.5,-8){\line(3,4){13}}
  \put(-28,-8){\line(3,2){27}}

  \put(0,-10){$a_2$}
  \put(0,-8){\line(-3,4){12.7}}
  \put(1.5,-8){\line(3,4){13}}

  \put(30,-10){$a_1$}
  \put(30,-8){\line(-3,4){12.7}}
  \put(29,-8){\line(-3,2){26}}

\end{picture}

\end{centering}

Corollary \ref{C:weak_zigzag} has an obvious cycle-version: every weak
zigzag cycle can be refined to a zigzag cycle, possibly starting and
ending in a new point.

\begin{definition}
  A lattice $V$ is \emph{simply connected} if there are no zigzag
  cycles.
\end{definition}

\begin{proposition}\label{P:simply_con}
  For a lattice $V$ the following are equivalent:
  \begin{enumerate}[a)]
  \item\label{P:simply_con:sc} $V$ is simply connected.
  \item\label{P:simply_con:comp} The endpoints of a zigzag of length
    at least $2$ are incomparable.
  \item\label{P:simply_con:infsup} For every zigzag $x_0\ldots x_m$ we
    have
    \[\inf(x_0,x_m)\leq x_i\leq\sup(x_0,x_m),\]
  \end{enumerate}
\end{proposition}
\begin{proof}\mbox{}\par
\noindent\ref{P:simply_con:infsup})$\to$\ref{P:simply_con:comp}) If
$x_0\ldots x_m$ is a zigzag and $x_0\leq x_m$, we have $x_i\leq x_m$
for all $i$ by assumption. Hence $m<2$ by the zigzag condition.\\

\noindent\ref{P:simply_con:comp})$\to$\ref{P:simply_con:sc}) A zigzag
cycle would have length $\geq 6$ and start and end in the same
point.\\

\noindent\ref{P:simply_con:sc})$\to$\ref{P:simply_con:comp}) Assume
that $x_0\ldots x_m$ is a zigzag with $x_0\leq x_m$ of minimal length
$m\geq 2$. We have observed above\footnote{Remark \ref{R:cycle}} that
$m\geq 5$. We know also that $m$ is odd, since, for example, $a_0\leq
a_{n+1}$ implies $a_0\leq b_n$. So, without loss of generality, our
sequence has the form $a_0\ldots b_n$, for $n\geq 2$. By minimality
$a_0\geq b_n$ is ruled out, so we have $a_0\leq b_n$. By minimality
and $n\geq 2$, we have that $a_0$ is not comparable with $b_{n-1}$ and
$a_n$, $b_n$ is not comparable with $b_0$ and $a_1$, and $a_n$ is not
comparable with $b_0$. This means that $a_0\ldots b_n,a_0$ is a weak
zigzag cycle, which can be refined to a zigzag cycle. So $V$ is not
simply connected.\\

\noindent\ref{P:simply_con:comp})$\to$\ref{P:simply_con:infsup}) By
duality it is enough to show $x_i\leq\sup(x_0,x_m)$. Let $c$ be any
element which bounds $x_0$ and $x_m$. We have to show that $x_i\leq c$
for all $i$. If this is not the case, we may shrink the sequence and
may assume that $m\geq 2$ and $x_i\not\leq c$ for all
$i=1,\ldots,m-1$. This implies that the sequence has the form
$a_0,b_0\ldots b_{n-1},a_n$. Since $b_0=\sup(a_0,a_1)$, we have $n\geq
2$. Now $c\not\geq b_0$, so by Lemma \ref{L:davorsetzen} either
$c,b_0\ldots a_n$ is a zigzag or $c,a_0,b_0\ldots a_n$ is a weak
zigzag with comparable endpoints, which can be refined to a zigzag.
Both zigzags have comparable endpoints.
\end{proof}
\begin{proposition}\label{P:construction_sc}
  Let $U$ be a simply connected sublattice of $V$. Let
  $(v_\alpha\mid\alpha<\lambda)$ be a wellordering of $V\setminus U$
  such that all $V_\beta=U\cup\{v_\alpha\mid\alpha<\beta\}$ are
  sublattices of $V$. Then $V$ is simply connected.
\end{proposition}
\begin{proof}
  Let $x_0\ldots x_n=x_0$ be a zigzag cycle in $V$. The cycle cannot
  be contained in $U$, so is is contained in some minimal
  $V_{\beta+1}$. If $x_i=v_\beta$, we have that $x_{i-1}$ and
  $x_{i+1}$ are in $V_\beta$. But then $x_i=\inf( x_{i-1},x_{i+1})$
  (or $\sup(x_{i-1},x_{i+1})$) is also in $V_\beta$, a contradiction.
\end{proof}

\section{The $N$-pseudospace}

An $N$-geometry is a structure $(V,<,\A_{-1},\A_0,\ldots,\A_{N+1})$,
where $(V,<)$ is a partial order with smallest element $\Null$ and
largest element $\Eins$, the \emph{layers} $\A_s$ form a partition
of $V$ and
\begin{enumerate}
\item if $x\in\A_s$, $y\in\A_t$, then $x<y$ implies $s<t$
\item $\Null\in\A_{-1}$ and $\Eins\in\A_{N+1}$.
\end{enumerate}
\noindent It follows that $\Null$ and $\Eins$ are the only elements of
$\A_{-1}$ and $\A_{N+1}$, respectively.\\

Let $a<b$ be elements of the $N$-geometry $V$ and $s$ the index of a
layer between the layers of $a$ and $b$. Then there is a unique
extension $V(x)=V\cup\{x\}$ such that $x\in\A_s$ and the relations
between $x$ and elements $c\in V$ are implied by $a<x<b$, i.e.
\begin{align*}
  c<x&\,\Leftrightarrow\;c\leq a\\
  x<c&\,\Leftrightarrow\;b\leq c
\end{align*}
We call such an extension \emph{simple} and the triple $(a,b,s)$ its
\emph{type}.\\

The \emph{countable universal} $N$-pseudospace $\Nps$ is an
$N$-geometry which is obtained as the union of a countable sequence
$V_0\subset V_1\subset\dotsb$ of simple extensions $V_{i+1}=V_i(x_i)$,
starting from $V_0=\{\Null,\Eins\}$, such that every possible type
$(a, b, s)$ with $a, b\in \Nps$ is used infinitely often. It is easy
to see that the resulting structure does not depend on the actual
choice of the sequence.

\begin{theorem}\label{T:vollstaendig}
  The complete theory of $\Nps$ is axiomatised by
  \begin{enumerate}
  \item $M$ is an $N$-geometry
  \item $M$ is a simply connected lattice
  \item For all $a<b$ in $M$ and every $s$ between the layers of $a$
    and $b$, there are infinitely many $x\in\A_s$ which lie between
    $a$ and $b$.
  \end{enumerate}
\end{theorem}

\noindent We call a model of this theory a \emph{free
  $N$-pseudospace}.\\

To show that $\Nps$ is a free $N$-pseudospace we need (the trivial
part of) the following lemma:

\begin{lemma}\label{L:simple_lattice}
  Let $V$ be an $N$-geometry which is a lattice, and $V'$ a proper
  extension of $V$ by a single element. Then $V'$ is a simple
  extension of $V$ if and only if $V'$ is a lattice and $V$ a
  sublattice of $V'$
\end{lemma}
\begin{proof}
  That simple extensions have that property is easy to check. For the
  converse note first that an $N$-geometry which is a lattice is
  actually a complete lattice. So, if $V$ is a sublattice of the
  lattice $V'$, and $V'=V\cup\{x\}$, then there is a largest $a\in V$
  below $x$ and a smallest $b\in V$ above $x$. This shows that $V'$ is
  a simple extension.
\end{proof}

Now let us show that $\Nps$ is a model of our axioms.
$V_0=\{\Null,\Eins\}$ is a simply connected lattice. So by Lemma
\ref{L:simple_lattice}, $\Nps$ is a lattice and all $V_i$ used in the
construction of $\Nps$ are sublattices. It follows from Proposition
\ref{P:construction_sc} that $\Nps$ is simply connected. The last
axiom scheme holds, since all possible types are used infinitely
often: let $a<b$ in $V_i$ and $s$ between the layers of $a$ and $b$.
Then for infinitely many $j\geq i$ the type of the extension
$V_{j+1}=V_j(x_j)$ is $(a,b,s)$.\\

The proof of the completeness of the axioms needs some preparations.
Let us work in a free $N$-pseudospace $M$ .

\begin{definition}
  A subset $A$ of $M$ is \emph{closed} if it contains $\Null$ and
  $\Eins$ and is closed under connecting elements by zigzags: if
  $x_0\ldots x_n$ is a zigzag and $x_0,x_n$ are in $A$, then all $x_i$
  are in $A$.
\end{definition}

For elements $a\leq b$ of $M$ define\footnote{We will use $[b,a]$ and
  $(b,a)$ as well, if it is convenient.} the set $[a,b]$ to be $\{x\in
M\mid a\leq x\leq b\}$, and $(a,b)$ to be $\{x\in M\mid a<x<b\}$. It
follows from Proposition \ref{P:simply_con} that
$\{\Null,\Eins\}\cup[a,b]$ is closed.

If $x$ and $y$ are not comparable, then $x,\sup(x,y),y$ and
$x,\inf(x,y),y$ are zigzags. This shows that closed subsets are
sublattices of $M$.\\

\begin{definition}
  Let $A$ be closed and $x\in M$. A \emph{direct path from $x$ to $A$}
  is a zigzag $x=x_0\ldots x_n$ such that $x_i\not\in A$ for all
  $i<n$, $x_n\in A$, and $(x_{n-1},x_n)\cap A=\emptyset$.

  If there is a direct path from $x$ to $A$ of length $n\geq 2$, we
  denote by $\delta(x,A)$ the minimal such $n$. Otherwise
  $\delta(x,A)$ is undefined.
\end{definition}

\begin{remark}\label{r:direct}
  If $A$ is a closed subset of $M$ and $x\in M\setminus A$, for any
  zigzag $x=x_0,x_1\ldots x_m$ which connects $x$ to an element $x_m$
  of $A$ we obtain a direct path from $x$ to $A$ as follows: let $x_n$
  be the first element which belongs to $A$. By replacing $x_n$ if
  necessary, we may assume that no element of $A$ lies strictly
  between $x_{n-1}$ and $x_n$ yielding a direct path from $x$ to $A$.
\end{remark}

\begin{lemma}\label{L:einfachclosed}
  Let $A$ be closed and $x\in M\setminus A$. Then $A\cup\{x\}$ is
  closed if and only if there there is no direct path from $x$ to $A$
  of length $\geq 2$.
\end{lemma}
\begin{proof}
  This is clear from the definition of closedness and the above
  construction of a direct path.
\end{proof}
\noindent Remember that if $A\cup\{x\}$ is closed, then it is a simple
extension of $A$ by Lemma~\ref{L:simple_lattice}. Note that, for
$x\not\in A$, there are exactly two direct paths of length~$1$
connecting $x$ to $A$, namely $x,\ab_{x,A}$ and $x,\ba_{x,A}$, where
\begin{align*}
  \ab_{x,A}&=\max\{y\in A\mid y\leq x\}\\
  \ba_{x,A}&=\,\min\{y\in A\mid x\leq y\}.
\end{align*}

\begin{lemma}\label{L:interval}
  A direct path from $x$ to a closed set $A$ of length $n\geq 2$ is
  contained in $(\ab_{x,A},\ba_{x,A})$.
\end{lemma}
\begin{proof}
  Let $x=x_0,x_1\ldots x_n$ be a direct path from $x$ to $A$ of length
  $n\geq 2$. If this path were not contained in
  $(\ab_{x,A},\ba_{x,A})$, it would follow from Proposition
  \ref{P:simply_con} that $x_n$ is not in $(\ab_{x,A},\ba_{x,A})$.
  Assume for example that $x_n\not <b=\ba_{x,A}$, i.e.\ $x_n\not\leq
  b$ since $x_n$ is not comparable with $x$. Choose an $i<n$ such that
  $b\geq x_i$, but $b\not\geq x_{i+1}$. Clearly $x_i$ is a sink in the
  zigzag $x_0,x_1\ldots x_n$. By Lemma \ref{L:davorsetzen} there are
  two cases.
  \begin{enumerate}
  \item $b<x_{i+1}$ and $b,x_{i+1}\ldots x_n$ is a zigzag. Since now
    $x<x_{i+1}$, we have $i=0$ and therefore $x_{i+1}\not\in A$. This
    contradicts the closedness of $A$.
  \item $b,x_i,x_{i+1}\ldots x_n$ is a weak zigzag. Here we have again
    two subcases: if $i<n-1$, we have $x_{i+1}\not\in A$ and, after a
    refinement, a contradiction to the closedness of $A$. Otherwise,
    if $i=n-1$, set $c=\inf(b,x_n)$, which lies in $A$. Then
    $x_{n-1}\leq c<x_n$. This contradicts the directness of our path.
  \end{enumerate}
\end{proof}
  We will use the following notation: if $\ab_{x,A}\in\A_s$ and
  $\ba_{x,A}\in\A_t$, we write $|x,A|=|t-s|$. Note that $x\not\in A$
  implies $|x,A|\geq 2$.
\begin{corollary}\label{C:eng_closed}
  If $A$ is closed and $|x,A|=2$, then $A\cup\{x\}$ is closed.
\end{corollary}
\begin{proof}
  For any zigzag $x=x_0, x_1,\ldots x_m$ from $x$ to some $x_m\in A$
  we obtain a direct path $x=x_0\ldots x_n'$ from $x$ to $A$ as in
  Remark~\ref{r:direct}. If $n\geq 2$ this direct path lies completely
  in $(\ab_{x,A},\ba_{x,A})$. So all $x_i$ would have to lie in the
  same layer by assumption. Hence $n\leq 1$. Starting from $x_1'$ we
  obtain a new zigzag to $x_m$. Since $A$ is closed all of its
  elements lie in $A$ showing $A\cup\{x\}$ to be closed as well.
\end{proof}
Let $A\subset B$ be closed subsets of $M$. We call $B$ \emph{finitely
  constructible} over $A$ if there is a sequence $v_1,\ldots v_k$ such
that $B=A\cup\{v_1,\ldots,v_k\}$ and all $B_i=A\cup\{v_1,\ldots,v_i\}$
are closed in $M$. We assume also that $v_i\not\in B_{i-1}$.
The following proposition corresponds to Prop.\ 2.19 in  \cite{kT14}.
\begin{proposition}\label{P:endlicher_abschluss}
  Let $A$ be closed in $M$, then any $x\in M$ is contained in a
  finitely constructible extension $B$ of $A$.
\end{proposition}
\begin{proof}
  We will prove the proposition by induction on $|x,A|$ and
  $\delta(x,A)$.\\

  If $A\cup\{x\}$ is closed, there is nothing to show. So, by Lemma
  \ref{L:einfachclosed}, $\delta(x,A)=n$ is defined. Let $x=x_0\ldots
  x_n$ be a direct path from $x$ to $A$. We will show first, that
  $x_1$ is contained in a finitely constructible extension $B'$ of
  $A$. For this we note that by Lemma \ref{L:interval}, we have
  $x_1\in(a,b)$, where $a=\ab_{x,A}$ and $b=\ba_{x,A}$, and whence
  $|x_1,A|\leq|x,A|$. We distinguish two cases
  \begin{enumerate}
  \item $n\geq 3$. Then $\delta(x_1,A)$ is defined and $\leq n-1$. So
    we can find $B'$ by induction.
  \item $n=2$. Then $x_1$ is in $(a,x_2)$ or in $(x_2,b)$. In either
    case $|x_1,A|<|x,A|$ and we find $B'$ by induction.
  \end{enumerate}
  Since $x$ is in $(a,x_1)$ or in $(x_1,b)$, we have $|x,B'|<|x,A|$.
  Now we can use induction to find $B\supset B'$.
\end{proof}
\begin{proposition}\label{P:existenz}
  If $M$ is  $\omega$-saturated and $A$ is a finite closed subset
  of $M$, then for any $a<b$ in $A$ and every $s$ between the layers
  of $a$ and $b$, there is a simple extension $A(x)$ of type $(a,b,s)$
  which is closed in $M$.
\end{proposition}
\begin{proof}
  Assume $a\in\A_r$ and $b\in\A_t$. We proceed by induction on
  $t-r$.\\

  \noindent If $t-r=2$, the third axiom scheme in the formulation of
  Theorem \ref{T:vollstaendig} yields an $x\in\A_s\setminus A$, which
  lies between $a$ and $b$. Clearly $a=\ab_{x,A}$, $b=a_{x,A}$ and
  $A\cup\{x\}$ is closed by Corollary \ref{C:eng_closed}.\\

  \noindent Now assume $t-r\geq 3$. Then $r+2\leq s$ or $s\leq t-2$.
  We consider only the second case, the first is dual to it. Let $a_0$
  be an arbitrary element of $\A_{r+1}\cap(a,b)$ and $A'$ a finite
  closed extension of $A$ which contains $a_0$. Now, using the
  induction hypothesis, we find a sequence of closed sets
  \[A'=A'_0\subset A_0\subset A'_1\subset A_1\subset\] such that
  $A_i=A'_i(b_i)$ is a simple extension of type $(a_i,b,t-1)$ and
  $A'_{i+1}=A_i(a_{i+1})$ is a simple extension of type $(a,b_i,r+1)$.
  Stop the construction at some arbitrary stage $k$. Then, using
  induction, choose a simple extension $A_k(x)$ of type $(a,b_k,s)$
  which is closed in $M$.

  The element $x$ lies in $\A_s$ and lies between $a$ and $b$. Since
  the only element in $A_k\cap(a,b)$ which is comparable with $x$ is
  $b_k$, we actually have $a=\ab_{x,A}$, $b=\ba_{x,A}$

  Consider a zigzag $x=x_0\ldots x_n$ which connects $x$ with an
  element of $A$ and lies completely in $(a,b)$. Since $A_k\cup\{x\}$
  is closed, $x_1$ lies in $A_k$. Since $x_1$ is comparable with $x$,
  we have $x_1=b_k$. Continuing this way we see that our sequence
  begins with $x,b_k,a_k\ldots b_0,a_0$, and therefore $n\geq 2(k+1)$.
  So, either $A\cup\{x\}$ is closed or $\delta(x,A)\geq 2(k+1)$. Since
  $k$ was arbitrary, the desired $x$ exists by the $\omega$-saturation
  of $M$.
\end{proof}
\begin{corollary}\label{C:back}
  Let $M$ be $\omega$-saturated containing a finite closed subset $A$.
  Assume that $A$ is also a closed subset of a finite simply connected
  $N$-geometry $B$. Then $M$ contains a closed subset $B'$ which is
  isomorphic to $B$ over $A$.
\end{corollary}
\begin{proof}
  The proof of Proposition \ref{P:endlicher_abschluss} shows that $B$
  is constructible over $A$. Write $B=A\cup\{v_1,\ldots,v_k\}$ where
  all $B_i=A\cup\{v_1,\ldots,v_i\}$ are closed in $B$ and $v_i\not\in
  B_{i-1}$. We define a sequence $\mathrm{id}_A=f_0\subset
  f_i\subset\dotsb f_k$ of isomorphisms between the $B_i$ and closed
  subsets $B'_i$ of $M$ as follows: if $f_i$ is defined, note that by
  Lemma \ref{L:simple_lattice} (applied to $B$) $B_{i+1}=B_i(v_{i+1})$
  is a simple extension, say of type $(a,b,s)$.
  Proposition~\ref{P:existenz} gives us a closed simple extension
  $B'_{i+1}=B'_i(v'_{i+1})$ of type $(f_i(a),f_i(b),s)$. Then
  $f_{i+1}$ can be defined by $f_{i+1}(v_{i+1})=v'_{i+1}$. Finally
  $f_k:B_k\to B'_k$ is the desired isomorphism.
\end{proof}
\begin{proof}[Proof of \ref{T:vollstaendig}]
  It remains to show that any two free $N$-pseudospaces $M$ and $M'$
  are elementarily equivalent.
  For this it suffices to show that if $M$ and $M'$ are
  $\omega$-saturated, then they are partially isomorphic. Indeed, by
  Proposition~\ref{P:endlicher_abschluss} and Corollary~\ref{C:back}
  the family of all isomorphisms $f:A\to A'$ between finite closed
  subsets of $M$ and $M'$ has the back and forth property. It is not
  empty because the subset $\{\Null,\Eins\}$ is closed in both models.
\end{proof}
\begin{corollary}\label{C:closed_type}
  Two closed sets have the same type if they are isomorphic. More
  precisely: any isomorphism between closed subsets of free
  $N$-pseudo\-spaces is an elementary map.\qed
\end{corollary}

\section{Independence}

\noindent In this section we work in a big saturated free
$N$-pseudospace $M$, (the monster model).\\

For a subset $A$ of $M$, we denote by $\cl(A)$ the \emph{closure} of
$A$, the smallest closed subset of $M$ which contains $A$. We will see
later, that $\cl(A)$ is finite if $A$ is finite and has the same
cardinality as $A$ otherwise.

\begin{definition}[cp.\ \cite{kT14}, Thm.\ 2.35]
  Let $A$, $B$, and $C$ be subsets of $M$. We say that $A$ and $C$ are
  \emph{independent over} $B$, if every zigzag $x_0\ldots x_n$ between
  an element of $A$ and an element of $C$ \emph{crosses} $B$,
  i.e.\ some $x_i$ belongs to $\cl(B)$ or some interval
  $(x_i,x_{i+1})$ intersects $\cl(B)$. We write this as $A\IND_BC$.
\end{definition}

We will show that the theory of free $N$-pseudospaces is
$\omega$-stable and that independence coincides with forking
independence. Stability (actually $\omega$-stability) follows easily
from Proposition \ref{P:endlicher_abschluss} and Corollary
\ref{C:closed_type} (but we will write out another proof below). For
the characterization of forking independence we need a series of
lemmas.

\begin{definition}\label{D:independent}
  Let $X$ be a subset and let $A$ be a closed subset of $M$. The
  \emph{gate} $\gate(X/A)$ of $X$ over $A$ is the set of all endpoints
  of direct paths from elements of $X$ to $A$.
\end{definition}

\begin{proposition}\label{P:ind_closed}
  Let $A$, $B$ and $C$ be closed and $B$ a subset of $A$ and $C$. Then
  the following are equivalent.
  \begin{enumerate}[a)]
  \item\label{P:ind_closed:ind} $A\IND_B C$
  \item\label{P:ind_closed:gate} $\gate(A/C)\subset B$
  \item\label{P:ind_closed:free}
    \begin{enumerate}[i)]
    \item  $A$ is \emph{free} from $C$ over $B$, i.e.\ if $a\in A$,
      $c\in C$ and $a\leq c$, there is a $b\in B$ such that $a\leq
      b\leq c$.
    \item $A\cup C$ is closed.
    \end{enumerate}
  \end{enumerate}
The equivalence between \ref{P:ind_closed:ind}) and
\ref{P:ind_closed:gate}) holds for arbitrary $A$: that $A$ is closed,
or $B$ is a subset of $A$ is not used.
\end{proposition}
\begin{proof}
  \noindent\ref{P:ind_closed:ind})$\to$\ref{P:ind_closed:free}):
  Assume $A\IND_BC$. If $a\in A$, $c\in C$ and $a\leq c$, then $a$ or
  $ac$ is a zigzag from $A$ to $C$, depending on whether $a=c$ or
  $a<c$. This shows that $A$ is free from $C$ over $B$. To see that
  $A\cup C$ is closed we have to show that every path $x_0\ldots x_n$
  from $A$ to $C$ lies completely in $A\cup C$. If some $x_i$ is in
  $B$, we know that $x_0\ldots x_i$ is in $A$ and $x_i\ldots x_n$ is
  in $C$. If $(x_i,x_{i+1})$ intersects $B$ in $b$, then $x_0\ldots
  x_i,b$ is a zigzag, which must be in $A$, and $b,x_i\ldots x_n$ is a
  zigzag, which must be in $C$.\\

  \noindent\ref{P:ind_closed:free})$\to$\ref{P:ind_closed:gate}): Let
  $x_0\ldots x_n$ be a direct path from $A$ to $C$. If $n=0$, we have
  $x_n=x_0\in C\cap A=B$. Otherwise, $x_{n-1}\in A\setminus C$, so
  $(x_{n-1},x_n)$ intersects $B$ in $b$. Since the path is direct,
  this can only happen if $x_n=b$.\\

  \noindent\ref{P:ind_closed:gate})$\to$\ref{P:ind_closed:ind}): Let
  $x_0\ldots x_n$ be a zigzag from $A$ to $C$. We want to show that it
  crosses $B$. For this we can assume that $x_n$ is the first element
  in $C$. If the zigzag is a direct path from $A$ to $C$, we have
  $x_n\in B$. Otherwise for some $c\in (x_{n-1},x_n)\cap C$, the path
  $x_0\ldots x_{n-1},c$ is direct from $A$ to $C$, which again implies
  $c\in B$.\\

  \noindent\ref{P:ind_closed:ind})$\to$\ref{P:ind_closed:gate}): Let
  $x_0\ldots x_n$ be a direct path from $x_0\in A$ to $C$. No interval
  $(x_i,x_{i+1})$ can contain an element $c$ of $C$: For $i=n-1$ this
  would directly contradict the directness. For $i<n-1$ the zigzag
  $c,x_{i+1}\ldots x_n$ would imply $x_{i+1}\in C$, again a
  contradiction. So $x_n$ is the only element of $C$ which the path
  can cross. It follows $x_n\in B$.
\end{proof}
\begin{lemma}\label{L:free_amalgam}
  \begin{enumerate}
  \item\label{L:free_almalgam:existenz} Assume that on $A$ and $C$ are
    partial orders which agree on their intersection $B$. Then these
    orders have a unique extension to $A\cup C$ such that $A$ is free
    from $C$ over $B$. We denote this partial order by $A\otimes_B C$.
    If $A$ and $C$ are $N$-geometries, with common subgeometry $B$,
    $A\otimes_BC$ has a natural structure of an $N$-geometry.
  \item\label{L:free_almalgam:connected} If the lattices $A$ and $C$
    are $N$-geometries and $B$ closed in $A$ and $C$. Then
    $A\otimes_BC$ is a lattice with closed subsets $A$ and $C$. If
    furthermore $A$ and $C$ are simply connected, then so is
    $A\otimes_BC$.
  \end{enumerate}
\end{lemma}
\begin{proof}
  \nc{\abc}{A\otimes_BC} \ref{L:free_almalgam:existenz}. is easy to
  check. For part \ref{L:free_almalgam:connected} we check first that
  $B$ and $C$ are sublattices of $\abc$. Indeed, if $c$ is a common
  upper bound of $a_1$ and $a_2$, there are $b_i$ with $a_i\leq
  b_i\leq c$. It follows that
  $a_i\leq\sup_A(b_1,b_2)=\sup_C(b_1,b_2)\leq c$. And whence
  $\sup_A(a_1,a_2)\leq c$. This shows that $\sup_A(a_1,a_2)$ is the
  supremum of $a_1,a_2$ in $\abc$. We have to show that two elements
  $a\in A$ and $c\in C$ have a supremum in $\abc$. Let $b_a$ be the
  smallest upper bound of $a$ in $B$, and $b_c$ be the smallest upper
  bound of $c$ in $B$. Set also $u=\sup_A(a,b_c)$ and
  $v=\sup_C(c,b_a$). Then $u$ is the smallest common upper bound of
  $a$ and $c$ in $A$, and $v$ the smallest common upper bound of $a$
  and $c$ in $C$. So we have to show that $u$ and $v$ are comparable.
  Consider the alternating sequence $b_a,a,u,b_c$. There are 7 cases:
  \begin{enumerate}
  \item\label{1} $b_a,a,u,b_c$ is a weak zigzag. This can then be
    refined to a zigzag $b_a,a',u,b_c$. Since $B$ is closed in $A$, we
    have $u\in B$. This implies $b_a\leq u$ and $v\leq\sup(c,u)=u$
  \item\label{2} $b_a\leq b_c$. Then $v\leq b_c\leq u$.
  \item\label{3} $b_c\leq b_a$. Then $u\leq b_a\leq v$
  \item $b_a\leq u$. Then $u=\sup(b_a,b_c)\in B$. This implies $v\leq
    u$ as in case \ref{1}.
  \item $u\leq b_a$. Then $u\leq v$.
  \item $a\leq b_c$. Then $b_a\leq b_c$. This is case \ref{2}.
  \item $b_c\leq a$. Then $b_c\leq b_a$. This is case \ref{3}.
  \end{enumerate}
  We have still to show that $A$ and $C$ are closed in $\abc$. So let
  $a,c_0\ldots c_n,a'$ be a zigzag with $a,a'\in A$ and $c_i\in
  C\setminus A$. If $b\in[a,c_0]$ and $b'\in[c_n,a']$, then zigzag
  $b,c_0\ldots c_n,b'$ shows that $B$ is not closed in $C$. Finally
  assume that $\abc$ contains a zigzag cycle $Z$. If $Z$ intersects
  $A$, it must be a subset of $A$, since $A$ is closed in $\abc$. Then
  $A$ would not be simply connected. Similarly $Z$ cannot intersect
  $C$. This is not possible.
\end{proof}

The following lemma is easy to check.
\begin{lemma}\label{L:splice}
  Consider two zigzags $\ldots a_n,b_n$ and $a'_0,b'_0\ldots$. Assume
  $a'_0=\inf(b_n,b'_0)$ and that $a_n$ and $a'_0$ are incomparable,
  then $\ldots a_n,b_n,a'_0,b'_0\ldots$ is a weak zigzag.\qed
\end{lemma}
\begin{lemma}\label{L:connect}
  Let $a_0\ldots x\ldots a_n$ be a zigzag. Assume $a_0\leq r$,
  $a_n\leq s$ and $x\not\leq r,s$. Then $x$ lies on zigzag between $r$
  and $s$.
\end{lemma}
\begin{proof}
  We may assume that $b_0\not\leq r$ and $b_{n-1}\not\leq s$. Now
  apply \ref{L:davorsetzen} to the beginning and to the end of the
  zigzag.
\end{proof}
\begin{lemma}
  Let $A$ be closed in $M$ and $X$ finite. Then $\cl(AX)\setminus A$
  is finite.
\end{lemma}
\begin{proof}
  By Proposition \ref{P:endlicher_abschluss} $X$ is contained in a
  finitely constructible extension of $A$.
\end{proof}
\begin{lemma}
  Let $a\ldots x\ldots b$ and $x,y\ldots c$ be two zigzags. Then
  $y\ldots c$ is a final segment of a zigzag from $a$ to $c$ or a
  final segment of a zigzag from $b$ to $c$
\end{lemma}
\begin{proof}
  We need to prove this only for the first three elements of
  $x,y\ldots c$. So we may assume the path is $x,y,c$. Also we may
  assume that $x<y$. In the zigzag $a\ldots x\ldots b$ let $a_1$ be
  the first and $b_1$ the last element $\leq y$. Set $r=\sup(a_1,c)$
  and $s=\sup(b_1,c)$, so $r,s \leq y$. There are three cases.
  \begin{enumerate}
  \item $r=y$. There are two subcases.
    \begin{enumerate}
    \item $a_1=a$. Then $a,y,c$ is a zigzag.
    \item Otherwise, let $b'$ be the predecessor of $a_1$ in $a\ldots
      x\ldots b$. Then $b'$ is a peak and $a_1$ a sink. Lemma
      \ref{L:davorsetzen} yields to subsubcases.
      \begin{enumerate}
      \item $y<b'$. Then $x<b'$, and whence $a_1=x$. So $c,y,x\ldots
        b$ is a zigzag.
      \item $b'$ and $y$ are incomparable. Then by Lemma
        \ref{L:splice} \[a\ldots b',a_1,y,c\] is a weak zigzag, which
        refines to a zigzag which leaves $y,c$ unchanged.
      \end{enumerate}
    \end{enumerate}
    \item $s=y$. Analogously.
    \item $r,s<y$. We show that this cannot happen. The assumption
      implies that $x\not\leq r,s$. For, $x\leq r$ (together with
      $c\leq r$) would imply that $y\leq r$. Since $a_1\leq r$ and
      $b_1\leq s$ we use Lemma \ref{L:connect} to conclude that $x$
      lies on a zigzag between $r$ and $s$. It follows that
      $\inf(r,s,)\leq x$. This is not possible since $c\leq\inf(r,s)$
      and $c\not\leq x$ (remember that $x,y,c$ is a zigzag).
  \end{enumerate}
\end{proof}
\begin{corollary}[cp.\ \cite{kT14}, Prop.\ 2.28, 2.30]\label{C:cl_zz}
  The closure of a set $X$ consists of\/ $\Null$, $\Eins$ and all
  zigzags between elements of $X$.
\end{corollary}
\begin{proof}
  Let $A$ be the set of all elements which lie on zigzags connecting
  elements of $X$. Let $x,P,x'$ be a zigzag which connects two
  elements of $x,x'\in A$. So $x$ lies on a zigzag which connects
  $a,b\in X$ and $x'$ lies on a zigzag which connects $a',b'\in X$.
  The lemma gives now a zigzag from, say, $a$ to $x'$ which contains
  $P,x'$, and again a zigzag from, say, $b'$ to $a$, which still
  contains $P$. This shows $P\subset A$.
\end{proof}
In the next lemma we make use of the following terminology. Let
$P=x_0\dots x_n$ be an alternating sequence, then the \emph{contents}
$\cont(P)$ is the set of elements of $P$ together with the union of
all intervals $(x_i,x_{i+1})$. Hence a sequence $P$ crosses $B$ if an
only if $\cont(P)$ intersects $B$. Note that
$\cont(P)\subset\cont(P')$ if $P$ is a refinement of $P'$.
\begin{lemma}
  Let $P$ be a zigzag from $a$ to $x$ and $x<c$. Then there is a
  zigzag $Q$ from a to $c$ such that $\cont(Q)\setminus\cont(P)\subset
  U$, where $U=\{b\mid b\leq c \text{ and } b\not<x\}$.
\end{lemma}
\begin{proof}
  Let $P$ be the zigzag $a=x_0\ldots x_n=x$. There are two cases.
  \begin{enumerate}
  \item $a\leq c$. Then we set $Q=a,c$. Assume $\cont(Q)\not\subset
    U$. Then $b<x$ for some $b\in\cont(Q)=[a,c]$. This implies $a<x$.
    So this can only happen if $P=a,x$, in which case
    $[a,c]\subset[a,x]\cup U$.
  \item There is $i$ such that $x_i\not\leq c$ and $x_{i+1}\leq c$. By
    Lemma \ref{L:davorsetzen} there are two cases.
    \begin{enumerate}
    \item $c<x_i$. Then $Q=x_0\ldots x_i,c$ is a zigzag. Then
      $\cont(Q)\setminus\cont(P)\subset[x_i,c]$. If $[x_i,c]$ is not a
      subset of $U$, we have $x_i<x$, which implies $x_i<c$, a
      contradiction.
    \item $Q'=x_0\ldots x_{i+1},c$ is a weak zigzag. Then
      $\cont(Q')\setminus\cont(P)\subset[x_{i+1},c]$. If $[x_{i+1},c]$
      is not a subset of $U$, we have $x_{i+1}<x$, which implies
      $i+2=n$. In this case
      \[\cont(Q')\setminus\cont(P)\subset[x_{i+1},c]\setminus
             [x_{i+1},x]\subset U\].
    \end{enumerate}
  \end{enumerate}
\end{proof}
\begin{corollary}\label{C:ind_cl}
  $A$ is independent from $C$ over $B$ if and only if $\cl(AB)$ is
  independent from $\cl(BC)$ over $\cl(B)$.
\end{corollary}
\begin{proof}
  It is clear from the definitions that we may assume that $B$ is
  closed, and that independence is antitone in $A$ and $C$. So, by
  Corollary \ref{C:cl_zz} is suffices to prove the following:
  \emph{Let $a$ be independent over $B$ from $c$. Then $a$ is
    independent over $B$ from every element $x$ on every direct path
    from $c$ to $B$.} Using induction on the length of the path, we
  may assume that $x$ is the second element of the direct path
  $c,x\ldots b$, and $x\not=b$. We want to show that any zigzag $P$
  from $a$ to $x$ crosses $B$. By duality we may assume that $x<c$.
  Choose a zigzag $Q$ from $a$ to $c$ as in the Lemma. Then $\cont(Q)$
  contains some $b'\in B$. We show that $b'$ is also contained in
  $\cont(P)$. If not, $b'$ would be in $U$ and different from $x$.
  Lemma \ref{L:davorsetzen} yields two cases: either $b',x\ldots b$ is
  a zigzag, or $b',c,x\ldots b$ is a weak zigzag. In either case we
  conclude that $x\in B$, a contradiction.
\end{proof}
\begin{lemma}
  The $\gate(X/A)$ of a finite set $X$ over a closed set $A$ is
  finite.
\end{lemma}
\begin{proof}
  Let $B$ be the closure of $X\cup A$ and $x_0\ldots x_n$ a direct
  path from some $x_0\in X\setminus A$ to $A$. Then
  $x_n\in\{a_{x_{n-1},A},b_{x_{n-1},A}\}$ It follows that
  \[\#\gate(X/A)\leq 2\cdot\#(B\setminus A).\]
\end{proof}
\noindent We will give a different bound in Proposition
\ref{P:gates_flags}
\begin{theorem}\label{T:forking}
  The theory of free $N$-pseudospaces is $\omega$-stable. Three
  subsets $A$, $B$, $C$ of $M$ are forking-independent if and only if
  they are independent in the sense of Definition~\ref{D:independent}.
\end{theorem}
\begin{proof}
  Since both notions of independence hold for $X$, $B$, $C$ if and
  only if they hold for all $X_0$, $B$, $C$, where $X_0$ is a finite
  subset of $A$, it is enough to prove the theorem for finite $X$.
  Also both notions hold for $X$, $B$, $C$ if and only if they hold
  for $X$, $B$, $BC$. This means that we may assume that $B\subset C$.
  Now, it remains to check the six properties of
  \cite[Theorem~8.5.10]{TZ12}.
  \begin{enumerate}[a)]
  \item (Invariance) Independence is invariant under automorphisms of
    $M$: This is clear.
  \item (Local Character) We will show that for every finite $X$ and
    every $C$ there is a finite $C_0\subset C$ such that $X
    \IND_{C_0}C$: Choose $C_0$ big enough such that $\cl(C_0)$
    contains the (finite) gate of $X$ over $\cl(C)$. Then by
    Proposition \ref{P:ind_closed} $X \IND_{\cl(C_0)}\cl(C)$. This
    implies\footnote{By the trivial direction of Corollary
      \ref{C:ind_cl}.} $X \IND_{C_0}C$.
  \item (Weak boundedness) Consider a finite set $X$ and an extension
    $B\subset C$. We will show that the number of extensions
    $\tp(X'/C)$ of $\tp(X/B)$ with $X' \IND_BC$ is bounded by the
    number of extensions\footnote{It follows from Corollary
      \ref{C:acldcl} below that this number is $1$.} of $\tp(X/B)$ to
    $\cl(B)$. Consider two finite sets $X'$ and $X''$ which have the
    same type over $\cl(B)$ and which are both independent from $C$
    over $B$. Then $\cl(X'B)$ and $\cl(X''B)$ have the same type over
    $\cl(B)$ and are independent from $C$ over $\cl(B)$ by
    Corollary\ \ref{C:ind_cl}. By Lemma \ref{L:free_amalgam} the two
    structures $\cl(X'B)\cup\cl(C)$ and $\cl(X''B)\cup\cl(C)$ are
    isomorphic over $\cl(C)$. Since they are closed in $M$, they have
    the same type.
  \item\label{existence} (Existence) Let $X$, and $B\subset C$ be
    given. Set $D=\cl(XB)$. With Lemma \ref{L:free_amalgam} choose a
    structure $D'\otimes_{\cl(B)}\cl(C)$ such that $D'$ is isomorphic
    to $D$ over $\cl(B)$. Now $D'\otimes_{\cl(B)}\cl(C)$ is simply
    connected, and $\cl(C)$ is closed in $D'$. By Corollary
    \ref{C:back} we may assume that $D'\subset M$. Write $D'$ as
    $\cl(X'B)$. Then $X'$ has the same type as $X$ over $B$ and is
    independent from $C$ over $B$ by Proposition \ref{P:ind_closed}.
    \item (Transitivity) Let $X$ be given and $B\subset C\subset D$ be
      given. Assume $X\IND_BC$ and $X\IND_CD$. Then $X\IND_B\cl(C)$
      and and $X\IND_C\cl(D)$ by Corollary \ref{C:ind_cl}. By
      Proposition \ref{P:ind_closed} this means
      $\gate(X/\cl(B))=\gate(X/\cl(C))=\gate(X/\cl(D))$. Whence $X$ is
      independent from $D$ over $B$.
    \item (Weak Monotony) If $A$ is independent from $C$ over $B$ and
      $C'$ is a subset of $C$, then $A$ is independent from $C'$ over
      $B$. This is clear from the definition.
  \end{enumerate}
\end{proof}

\section{Gates}
\noindent In this section we work, as before in a big saturated free
$N$-pseudospace $M$.

\begin{proposition}[cp.\ \cite{kT14}, Prop.\ 2.28, 2.30]
  $\cl(X)$ is the algebraic closure of~$X$.
\end{proposition}
\begin{proof}
  If $X$ is finite, its closure is finite by Proposition
  \ref{P:endlicher_abschluss}, so contained in the algebraic closure
  of $X$. The closure of an arbitrary set $X$ is the union of the
  closures of its finite subsets, and so also contained in $\acl(X)$.
  Let $B$ be closed and $x\not\in B$. We will show that $x$ is not
  algebraic over $B$. Let $C$ be an arbitrary (small) extension of
  $B$. The existence part of the proof Theorem \ref{T:forking} yields
  a realization of $\tp(x/B)$ which is not an element of $C$. Taking
  $C=\acl(B)$, we see that $x$ is not algebraic over $B$.
\end{proof}
\begin{lemma}
  Let $A$ be a closed set and $z\in M\setminus A$. Then
  \[\etag(z/A)=\{x_1\mid z,x_1\ldots x_n \text{ a direct path from
    $x$ to $A$}\}\] is a \emph{flag}, i.e.\ a linearly ordered subset
  of $M$.
\end{lemma}
\begin{proof}
  Consider two direct paths from $z$ to $A$: $z,x_1\ldots x_m$ and
  $z,y_1\ldots y_n$. We want to show that $x_1$ and $y_1$ are
  comparable. So we may assume that $z$ is a peak (or a sink) in both
  paths. If $x_1$ and $y_1$ were not comparable, $y_n\ldots
  y_1,z,x_1\ldots x_m$ is a weak zigzag by Lemma \ref{L:splice}. It
  follows that $x_1$ and $y_1$ are in $A$. So $m=n=1$ and
  $x_1=y_1=\ab_{x,A}$. A contradiction.
\end{proof}
\begin{corollary}[cp.\ \cite{kT14}, Cor.\ 2.21]
  A direct path $z,x_1\ldots x_n$ from $z$ to $A$ is determined by the
  sequence of $s_1\dots s_{n-1}$ of the layers of the $x_i$.\qed
\end{corollary}
Remember that there are only finitely many direct paths from $z$ to
$A$ by Proposition \ref{P:endlicher_abschluss}.
\begin{corollary}\label{C:acldcl}
  $\cl(X)$ is the definable closure of $X$.
\end{corollary}
\begin{proof}
  A zigzag $x_0\ldots x_n$ is a direct path from $x_0$ to
  $\{\Null,x_n,\Eins\}$, and therefore definable from $x_0$ and $x_n$.
\end{proof}
\begin{proposition}[cp.\ \cite{kT14} Prop.\ 2.31, Cor.\ 2.33]
  \label{P:gates_flags}
  Gates of single elements are flags.
\end{proposition}
\begin{proof}
  Let $A$ be closed. We consider direct paths $x_0\ldots x_m$ and
  $y_0\ldots y_n$ from $x_0$ and $y_0$ to $A$. We claim that $x_m$ and
  $y_n$ are comparable if $x_0$ and $y_0$ are. Using induction on
  $m+n$ we may assume that $x_i$ is incomparable with $y_j$ if
  $(i,j)\not=(0,0)$. We will show that this implies $m=n=0$, for which
  the claim is obvious.

  So assume $m>0$ and $x_0\leq y_0$. Then the sequence $x_m\ldots
  x_1,x_0,y_0\ldots y_n$ is a weak zigzag. This can be refined to a
  zigzag, where only $x_0$ and $y_0$ are changed to $x'_0$ and $y'_0$.
  This new zigzag must be a subset of $A$. It follows that $x_1\in A$,
  and therefore $m=1$. But now $x'_0$ belongs to $A$ and is in
  $(x_1,x_0)$. This is impossible.
\end{proof}
\begin{corollary}
  $\#\gate(x/A)\leq N+3$\qed
\end{corollary}
\begin{lemma}\label{grenzen_gate}
  Let $A$ be closed and $x=x_0\ldots x_n$ a direct path from $x$ to
  $A$. Then all $\ab_{x_i,A}$ and $\ba_{x_i,A}$ for $i<n$ belong to
  $\gate(x/A)$.
\end{lemma}
\begin{proof}
  Assume $i<n$. We show that $a=\ab_{x_i,A}$ is in $\gate(x/A)$. Let
  $j\leq i$ be the first index with $a\leq x_j$. We note first that
  $(x_j,a)$ does not intersect $A$. Indeed, if $a<a'<x_j$, Lemma
  \ref{L:interval} gives us $a<a'<x_i$, which contradict the
  definition of $a$. There are two cases:
  \begin{enumerate}
  \item $j=0$. Then $x,a$ is a direct path from $x$ to $A$.
  \item $j>0$. Then we can can apply\footnote{Actually the dual of the
    lemma} Lemma \ref{L:davorsetzen}. The first case of the lemma
    cannot occur, since $x_{j-1}<a$ would imply that $x_i<a$ by Lemma
    \ref{L:interval}. This is impossible. So $x\ldots x_j, a$ is a
    weak zigzag. This refines to a direct path $x\ldots x'_j, b$ from
    $x$ to $A$, since $(x_j,a)$ does not intersect $A$.
  \end{enumerate}
\end{proof}

Let $A\subset B$ be closed subset of $M$, and $B$ finitely
constructible over $A$ via the sequence $v_1\ldots v_k$. Let
$(a_i,b_i,s_i)$ be the type of the extension
\[A(v_1,\ldots,v_i)/A(v_1,\ldots,v_{i-1}).\] We define the
\emph{boundary} of $B/A$ as
\[\boundary(B/A)=\{a_1,b_1,\ldots,a_k,b_k\}\cap A.\]
It will follow from the proof of the next proposition, that
$\boundary(B/A)$ does not depend on the choice of the constructing
sequence $v_1\ldots v_k$.
\begin{proposition}
  Let $A$ be closed and $X$ be finite and disjoint from $A$. Then
  \[\gate(X/A)=\boundary(\cl(XA)/A).\]
\end{proposition}
\begin{proof}
  Let $B=\cl(XA)$ be constructed via $v_1\ldots v_k$, and
  $(a_i,b_i,s_i)$ the corresponding types. We will show that both
  sets, $\gate(X/A)$ and $\boundary(B/A)$, are equal to
  \[G=\{\ab_{x,A},\ba_{x,A}\mid x\in B\setminus A\}.\]

  We show first, that $\boundary(B/A)=G$. Consider for example an
  element $b_i$ in the the boundary, i.e.\
  \[b_i=\ba_{v_i,A(v_1,\ldots,v_{i-1})}.\]
  Then $b_i\in A$ implies $b_i=\ba_{v_i,A}\in G$. If conversely
  $b_{x,A}$ is in $G$, let $v_i$ be maximal\footnote{In the sense of
    the partial order of $M$} in $[x,b_{x,A}]$. Then $v_i<b_i\leq
  b_{x,A}$ and maximality implies $b_i\in A$, which again implies
  $b_i=b_{x,A}\in A$ and that $b_i\in\boundary(B/A)$.

  An element of $\gate(X/A)$ is the last element $x_n$ of a direct
  path $x_0\ldots x_n$ from some $x_0\in X$ to $A$. But
  $x_n\in\{\ab_{x_{n-1},A},\ba_{x_{n-1},A}\}$. This shows
  $\gate(X/A)\subset G$. Let conversely $\ba_{x,A}$ be an element of
  $G$. Since $x\in B\setminus A$, we have $x=x_i$ for some direct path
  $x_0\ldots x_n$ from $x_0\in X$ to $A$ and $i<n$. By Lemma
  \ref{grenzen_gate} we have $\ba_{x,A}\in\gate(x_0/A)$.
\end{proof}


\hfill S.A.G

\noindent\parbox[t]{15em}{
Katrin Tent,\\
Mathematisches Institut,\\
Universit\"at M\"unster,\\
Einsteinstrasse 62,\\
D-48149 M\"unster,\\
Germany,\\
{\tt tent@uni-muenster.de}}
\hfill\parbox[t]{18em}{
Martin Ziegler,\\
Mathematisches Institut,\\
Albert-Ludwigs-Universit\"at Freiburg,\\
79104 Freiburg,\\
Germany,\\
{\tt ziegler@uni-freiburg.de}
}

\end{document}